\RequirePackage{fix-cm}
\documentclass[english]{amsart}

\usepackage{graphicx}
\usepackage[all]{xy}
\usepackage {float}
\usepackage{amsthm}
\usepackage{amsmath}
\usepackage{amssymb}

\usepackage[T1]{fontenc}
\usepackage[latin9]{inputenc}
\usepackage{float}
\usepackage{amstext}
\usepackage{fixltx2e}
\usepackage{graphicx}
\usepackage[all]{xy}

\numberwithin{equation}{section}
\numberwithin{figure}{section}
  \theoremstyle{plain}
  \newtheorem*{theorem*}{\protect\theoremname}
\theoremstyle{plain}
\newtheorem{theorem}{\protect\theoremname}[section]
  \theoremstyle{definition}
  \newtheorem{definition}[theorem]{\protect\definitionname}
  \theoremstyle{remark}
  \newtheorem*{remark*}{\protect\remarkname}
 \theoremstyle{definition}
 \newtheorem*{definition*}{\protect\definitionname}
  \theoremstyle{definition}
  \newtheorem{example}[theorem]{\protect\examplename}
  \theoremstyle{plain}
  \newtheorem{proposition}[theorem]{\protect\propositionname}
  \theoremstyle{remark}
  \newtheorem{remark}[theorem]{\protect\remarkname}
  \theoremstyle{plain}
  \newtheorem{lemma}[theorem]{\protect\lemmaname}
  \theoremstyle{plain}
  \newtheorem{corollary}[theorem]{\protect\corollaryname}
  \theoremstyle{plain}
  \newtheorem*{proposition*}{\protect\propositionname}

\usepackage{hyperref}
\usepackage{xcolor}
\hypersetup{colorlinks=false,linkbordercolor=red,linkcolor=green,pdfborderstyle={/S/U/W 1}
}
\makeatother

\usepackage{babel}
  \providecommand{\corollaryname}{Corollary}
  \providecommand{\definitionname}{Definition}
  \providecommand{\examplename}{Example}
  \providecommand{\lemmaname}{Lemma}
  \providecommand{\propositionname}{Proposition}
  \providecommand{\remarkname}{Remark}
  \providecommand{\theoremname}{Theorem}
\providecommand{\theoremname}{Theorem}

\begin{document}

\title{The Brill-Noether rank of a tropical curve
}

\author{Yoav Len
}

\begin{abstract}
We construct a space classifying divisor classes of a fixed degree
on all tropical curves of a fixed combinatorial type and show that
the function taking a divisor class to its rank is upper semicontinuous.
We extend the definition of the Brill-Noether rank of a metric graph
to tropical curves and use the upper semicontinuity of the rank function
on divisors to show that the Brill-Noether rank varies upper semicontinuously
in families of tropical curves. Furthermore, we present a specialization
lemma relating the Brill-Noether rank of a tropical curve with the
dimension of the Brill-Noether locus of an algebraic curve.

\end{abstract}

\maketitle

\section{Introduction}
\label{intro}

Let $\Gamma$ be a tropical curve in the sense of \cite{CaporasoAmini},
so that $\Gamma=(G,w,\ell)$ is a connected compact metric graph with a non-negative
integer weight function on its vertices. For integers $d$ and $r$,
the Brill-Noether locus of $\Gamma$, denoted $W_{d}^{r}(\Gamma)$,
is the set of equivalence classes of divisors of degree $d$ and rank
at least $r$. The Brill-Noether rank $w_{d}^{r}(\Gamma)$ of the
curve $\Gamma$ is the largest number $\rho$ such that every effective
divisor of degree $r+\rho$ is contained in an effective divisor of
degree $d$ and rank at least $r$. The main result in this paper
is the following: 
\begin{theorem*} (\ref{BN rank is USC})
The Brill-Noether rank is upper semicontinuous
on {\em $M_g^\text{tr}$\em}, the moduli space of tropical curves
of genus $g$. 
\end{theorem*}
Moreover, we extend the Brill-Noether Rank Specialization Lemma (\cite[Theorem 1.7]{LPP})
to the setting of tropical curves:
\begin{theorem*}
(\ref{Specialization BN})
 Let $X$ be a smooth projective curve over
a discretely valued field with a regular semistable model whose special
fiber has weighted dual graph $\Gamma$. Then for every $r,d\in\mathbb{N}$,
$\dim W_{d}^{r}(X)\leq w_{d}^{r}(\Gamma)$.
\end{theorem*}

The notion of Brill-Noether rank was introduced in \cite{LPP} in
the context of metric graphs and was motivated by the fact that the
function taking a metric graph to the dimension of its Brill-Noether
locus is not upper semicontinuous on the moduli space of metric graphs
of a fixed genus. This is in contrast with the analogous statement
for algebraic curves. However, as shown in Theorem 1.6 of the same
paper, the function taking a metric graph to its Brill-Noether rank
is indeed upper semicontinuous which implies that this notion serves
as a good tropical analogue for the dimension of the Brill-Noether
locus of an algebraic curve. 

The generalization of this upper semicontinuity to the setting of
tropical curves is not just a simple extension of the metric graph
case. The difficulty lies in the fact that, for a converging sequence
of curves in $M_{g}^{\text{{tr}}}$, the limit may be a curve with
an entirely different topology. To overcome this difficulty, we study
the divisor theory of tropical curves as the diameters of certain
subcurves become sufficiently small. See Proposition \ref{Arrangement single graph}
for a precise statement. Proposition \ref{Arrangement multiple graphs}
then shows that every divisor class contains an effective representative
whose restriction to these subcurves has a sufficiently high degree.
It follows that the rank of the divisor does not drop in the limit,
despite the change in topology.

In order to manage sequences of pairs of tropical curves and divisors,
we construct in Section \ref{Section: Universal Jacobian} a universal
Jacobian classifying divisors on all tropical curves of a fixed combinatorial
type, with the property that the fiber over each point is the Jacobian
of the corresponding curve. Within this space we will identify the
universal Brill-Noether locus and show in section \ref{sec:The-universal-Brill-Noether}
that it is a closed set, or equivalently, that the function taking
a divisor class to its rank is upper semicontinuous. We will then
use this result in section \ref{Section: BN rank } to conclude that
the Brill-Noether rank is upper semicontinuous on the moduli space
of tropical curves of genus $g$. 

\section{\label{section: preliminaries}Preliminaries}

In this section we briefly review some of the basics of divisor theory
on tropical curves. We refer the reader to \cite{CaporasoAmini}, \cite{BakerNorine},
\cite{Luo} and \cite[section 2]{LPP} for further details. We also
give a mild generalization of Luo's theory of rank determining sets
from metric graphs to tropical curves.

\subsection{Tropical curves}
\begin{definition}
By a metric graph we mean a pair $(G,\ell)$ such that $G$ is a connected
graph and 
\[
\ell:E(G)\to\mathbb{R}_{>0}
\]
is a function known as the length function. A \emph{tropical curve}
$\Gamma$ is a triple $(G,w,\ell)$ such that $(G,l)$ is a metric
graph and 
\[
w:V(G)\to\mathbb{Z}_{\geq0}
\]
is a function known as the weight function. A \emph{pure tropical
curve }is a tropical curve whose weight function is identically zero.
To emphasize that a curve is not pure we will sometimes call it a
weighted tropical curve.
\end{definition}

\begin{remark}
We implicitly identify a tropical graph $\Gamma$ with the underlying
metric space of the metric graph $(G,\ell)$.
\end{remark}

\subsection{Divisor theory of a tropical curve}\label{prelim}
Let $\Gamma=(G,0,l)$ be a pure tropical curve. Recall that a
\emph{rational function} on $\Gamma$ is a piecewise linear function
with integer slopes. To a rational function $f$, we associate a divisor
$\mbox{div}(f)$ whose value at each point $p$ of $\Gamma$ is the sum
of the incoming slopes of $f$ at $p$. The group of divisors of the form $\text{div}(f)$ is denoted $\text{Prin}(\Gamma)$. The $\emph{rank}$ of a divisor $D$ is
the smallest number $r$ such that for every effective divisor $E$ of
degree $r$, D is linearly equivalent to a divisor containing $E$
(where two divisors are linearly equivalent whenever their difference is
$\mbox{div}(f)$ for some rational function $f$).

Now let $\Gamma=(G,w,\ell)$ be any tropical curve. The pure tropical curve
associated to $\Gamma$ is the curve $\Gamma^{0}=(G,0,\ell)$.
For $\epsilon>0$ we denote $\Gamma_{\epsilon}^{w}=(G^{w},0,l_{\epsilon}^{w})$
the pure tropical curve obtained from $\Gamma^{0}$ by attaching $w(v)$
loops of length $\epsilon>0$ at every vertex $v$ of $\Gamma$. As
shown in \cite[Theorem 5.4]{CaporasoAmini}, for a divisor $D$ supported
on $\Gamma^{0}$, the rank of $D$ as a divisor on the curve $\Gamma_{\epsilon}^{w}$
is independent of the choice of $\epsilon$.
\begin{definition}
\label{weighted rank}(\cite{CaporasoAmini}) Let $\Gamma$ be a tropical
curve.
\begin{enumerate}
\item The divisor group of $\Gamma$ is defined to be the divisor group
of the underlying pure tropical curve, namely $\mbox{Div}(\Gamma):=\mbox{Div}(\Gamma^{0})$.
\item Let $D$ be a divisor on $\Gamma$. The \emph{rank} of $D$, denoted
$r_{\Gamma}(D)$, is its rank as a divisor on $\Gamma_{\epsilon}^{w}$
for any $\epsilon>0$. Since $r_{\Gamma_{\epsilon}^{w}}$ is independent
of $\epsilon$, this is well defined. When the curve $\Gamma$ is
known from context we shall simply write $r(D)$.
\end{enumerate}
\end{definition}

The \emph{genus }of a graph $G$ is its first Betti number $g(G)=|E|-|V|+1$,
and the genus of the tropical curve $\Gamma=(G,w,\ell)$ is 
\[
g(\Gamma)=g(G)+\sum_{v\in V(G)}w(v)\mbox{.}
\]

\subsection{Rank determining sets for tropical curves}\label{RDS}

The theory of rank determining sets, introduced by Y. Luo, shows that
in order to determine the rank of a divisor $D$ on a (non-weighted) metric graph,
one only needs to determine whether or not $D-E$ is linearly equivalent
to an effective divisor for a finite set of divisors $E$. We recall
the basic notions of the theory.
\begin{definition}
(\cite{Luo}) Let $A$ be a subset of a metric graph $\Gamma=(G,\ell)$
and let $D$ be a divisor on $\Gamma$. The $\mbox{A-}\emph{rank}$ of $D$,
denoted $r_{\Gamma,A}(D)$, is the largest number $r$ such that whenever
$E$ is an effective divisor of degree $r$ which is supported on
the set $A$, the divisor $D-E$ is linearly equivalent to an effective
divisor. The set $A$ is said to be a \emph{rank determining set
}if the rank of every divisor $D$ on $\Gamma$ coincides with its
$\mbox{A-rank}$. 
\end{definition}
For instance, if $G$ is a loopless graph then $V(G)$ is a rank determining
set for $\Gamma$ \cite[Theorem 1.6]{Luo}.
In Corollary \ref{prop:weighted RDS}, we extend the notion of rank determining sets to
tropical curves and show the existence of finite rank determining
sets supported on the underlying pure tropical curve. 

Next, we wish to give an equivalent definition of the rank in terms of the underlying pure tropical curve.

Let us first
consider an example:
\begin{example}
\label{Example - Rg} Let $R_{g}$ be the tropical curve consisting
of a single vertex $v$ of weight $g$ and no edges. The pure tropical
curve $(R_{g})_{\epsilon}^{w}$, obtained by attaching $g$ loops
to the vertex $v$, is known as the ``rose with $g$ petals'' (see Figure \ref{Rose}). 

\begin{figure}[H]
\centering 
\includegraphics[scale=0.5]{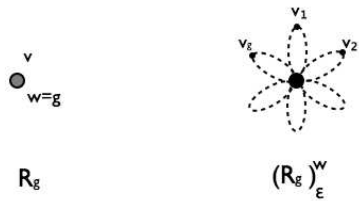}
\caption{The rose with $g$ petals} \label{Rose}
\end{figure}

\noindent Let $D$ be the divisor $d\cdot v$ on $R_{g}$ for some
$d\geq0$. Then 
\[r(D)=\begin{cases}
d-g & d>2g\\
\lfloor\frac{d}{2}\rfloor & d\leq2g
\end{cases}\mbox{.}
\]
In particular, $r(D)\geq r$ if and only if $d\geq r+\min\{r,g\}$.
If $d>2g$ then this follows immediately from the Riemann-Roch Theorem for
tropical curves (\cite[Theorem 5.4]{CaporasoAmini}). The case where
$d<2g$ appears in \cite[Lemma 3.7]{CaporasoAmini} for weighted non-metric
graphs and the proof in general is similar: let $v_{1},\ldots,v_{g}$
be the vertices in the middle of the attached loops of $(R_{g})_{1}^{w}$.
Then $v,v_{1},\ldots,v_{g}$ is a rank determining set for $(R_{g})_{1}^{w}$.
Let $E=b\cdot v+b_1\cdot v_1+\ldots+b_g\cdot v_g$ be an effective divisor of degree $\lfloor\frac{d}{2}\rfloor$ supported on $v,v_{1},\ldots,v_{g}$.
For each $i$, identify the interior of the loop containing $v_{i}$
with the interval $(0,1)$. Now define a rational function $f_i$ as follows
\[
f_{i}=\begin{cases}
x & 0<x\leq\frac{1}{2}\\
1-x & \frac{1}{2}<x<1
\end{cases}
\]
and extend $f_{i}$ to the rest of $(R_{g})_{1}^{w}$ by zero. Then
$D+\sum_{i}b_i\cdot\mbox{div}(f_{i})$ is effective and contains $E$, Hence
$r(D)\geq\lfloor\frac{d}{2}\rfloor$. For the converse, $ $$D$ could
either be a special divisor (i.e. $r>d-g$) or a non-special divisor.
If $D$ is a non-special divisor then $2r=2d-2g<2d-d=d$. Otherwise,
$D$ is special and $r\leq2d$ by Clifford's Theorem (\cite[Theorem 1]{Facchini}).
\end{example}
Roughly speaking, Example \ref{Example - Rg} shows that a single
vertex of weight $w(v)$ places $r+\min\{r,w(v)\}$ conditions on
a divisor of rank $r$. In general, for a divisor $E=b_{1}v_{1}+\cdots+b_{s}v_{s}$
on a tropical curve, we write $E^{*}$ for the divisor 
\[
E^{*}=\sum_{i=1}^{s}(b_{i}+\min\{b_{i},w(v_{i})\})v_{i}\mbox{.}
\]

\begin{remark}
A similar idea already appeared in the proof of the Specialization
Lemma for weighted graphs (\cite[Theorem 4.9]{CaporasoAmini}). To
show that the divisor obtained by specialization has the desired rank
it was first proved that for a fixed vertex $v$, its equivalence
class contains a divisor whose value at $v$ is at least $r+\min\{r,w(v)\}$.
\end{remark}

\begin{proposition}
\label{internalRank}Let $D$ be a divisor on a tropical curve $\Gamma=(G,w,\ell)$. Then $D$ has rank $r$ if and only if for every effective divisor $E$ of degree $r$, the divisor $D-E^*$ is linearly equivalent to an effective divisor. 

Moreover, if $A$ is a rank determining set for $\Gamma_0$ which contains the vertices with non-trivial weights, then it suffices to only consider divisors $E$ that are supported on $A$.\end{proposition}

\begin{proof}
To prove this we will use the following characterization of the rank
of a divisor on a tropical curve: 
\begin{equation}\label{ABrank}
r(D)=\underset{0\leq F\leq\mathcal{W}}{\min}(\deg(F)+r_{\Gamma_{0}}(D-2F)),
\end{equation}
where $\mathcal{W}=\sum_{v\in V(G)}w(v)\cdot v$ (\cite[Corollary 5.5]{AminiBaker}).

Suppose first that $D-E^*$ is linearly equivalent to an effective divisor whenever $E$ is effective of degree $r$ supported on $A$, and let us show that the rank of $D$ is at least $r$. 
By (\ref{ABrank}), we need to show that whenever $F$ is an effective divisor which is contained in $\mathcal{W}$, then $\deg(F)+r_{\Gamma_{0}}(D-2F)\geq r$. Let $F$ be such a divisor. If $\deg(F)>r$, then $\mbox{deg}(F)+r_{\Gamma_{0}}(D-2F)\geq r+1-1=r$.
Therefore we may assume that $\deg(F)\leq r$. Since $0\leq F\leq\mathcal{W}$,
$F^{*}$ is just $2F$. Let $F'$ be any effective divisor of degree
$r-\deg(F)$ supported on $A$. Since $F'+2F\leq(F'+F)^{*}$ and
$\mbox{deg}(F'+F)=r$, together with the fact that $F,F'$ are supported on $A$, we conclude that $D$ is linearly equivalent to an effective divisor that contains $F'+2F$. Hence $D-2F$ is equivalent to a divisor
containing $F'$. Since this is true for every effective divisor $F'$
of degree $r-\deg(F)$ which is supported on a rank determining set for $\Gamma_0$, we see that $r_{\Gamma_{0}}(D-2F)\geq r-\deg F$,
hence $\mbox{deg}(F)+r_{\Gamma_{0}}(D-2F)\geq r$. 

Conversely, suppose that the rank of $D$ is $r$, and let $E$ be an effective divisor of degree $r$ on $\Gamma_0$. Write $E=E_{w}+E_{0}$, where $E_{w}(v)=\min\{E(v),w(v)\}$
at every vertex $v$, zero everywhere else, and $E_{0}=E-E_{w}$. Then $E^{*}=2E_{w}+E_{0}$.
To prove that $D-2E_{w}-E_{0}$ is equivalent to an effective divisor
it suffices to show that its rank on $\Gamma_{0}$ is non-negative.
From the definition of the rank, subtracting an effective divisor
may only decrease the rank by the degree of the subtracted divisor,
hence $r_{\Gamma_{0}}(D-2E_{w}-E_{0})\geq r_{\Gamma_{0}}(D-2E_{w})-\deg(E_{0})$.
The fact that $E_{w}\leq\mathcal{W}$ implies $r\leq\mbox{deg}(E_{w})+r_{\Gamma_{0}}(D-2E_{w})$,
hence $r_{\Gamma_{0}}(D-2E_{w})-\deg(E_{0})\geq r-\mbox{deg}(E_{w})-\mbox{deg}(E_{0})=0$.
This is true for every such $E$, therefore $r_{\Gamma,V(G)}^{*}(D)\geq r$.
\end{proof}

In light of Proposition \ref{internalRank}, we introduce the following definition, which is a slight generalization of the notion of rank determining sets for tropical curves.

\begin{definition}
Let $A$ be a subset of a tropical curve $\Gamma$ and let $D$ be
a divisor on $\Gamma$. The \emph{weighted }$A\mbox{-rank}$ of
$D$, denoted $r_{\Gamma,A}^{*}(D)$ is the largest number $r$ such
that whenever $E$ is an effective divisor of degree $r$ which is
supported on $A$, the divisor $D-E^{*}$ is linearly equivalent (as
a divisor on $\Gamma^{0}$) to an effective divisor. We say that $A$
is a \emph{weighted rank determining set} if the rank of every divisor
$D$ on $\Gamma$ coincides with its weighted $A\mbox{-rank}$.
\end{definition}

An immediate corollary of Proposition \ref{internalRank}, is:

\begin{corollary}
\label{prop:weighted RDS}Let $\Gamma=(G,w,\ell)$ be a tropical curve. Then every rank determining set for $\Gamma_0$ which contains the vertices of positive weight is a weighted rank determining set for $\Gamma$.

In particular, the set of vertices of a tropical curve is a weighted rank determining set.
\end{corollary}
\qed

\section{\label{Section: Universal Jacobian}The universal Jacobian}

Let $\Gamma$ be a tropical curve of genus $g$. The \emph{Jacobian}
of $\Gamma$ is the group of divisor classes of degree zero on $\Gamma$:

\[
J(\Gamma)=\mbox{Div}^{0}(\Gamma)/\mbox{Prin}(\Gamma)\mbox{.}
\]
The Jacobian is identified with the torus $H_{1}(\Gamma^{0},\mathbb{R})/H_{1}(\Gamma^{0},\mathbb{Z})\simeq\mathbb{R}^{g(\Gamma^{0})}/\mathbb{Z}^{g(\Gamma^{0})}$,
where $g(\Gamma^{0})$ is the genus of $\Gamma^{0}$. See
\cite{BakerFaber}, \cite{MikhalkinZharkov}
for details. In this section, we introduce the universal space of
divisor classes over the space of all tropical curves of a fixed combinatorial
type.
\begin{definition}
Let $\Gamma=(G,w,\ell)$ be a tropical curve. The \emph{combinatorial
type} of $\Gamma$ is $(G,w)$.
\end{definition}
Fix a combinatorial type $(G,w)$. Let $E(G)=\{e_{1},...,e_{n}\}$
and denote $\sigma=\mathbb{R}_{\geq0}^{n}$. We can identify any point
$(s_{1},...,s_{n})$ in the interior of $\sigma$ with the tropical
curve of type $(G,w)$ with $\ell(e_{i})=s_{i}$ which we denote by
$\Gamma_{(s_{1},...,s_{n})}$. Next, we extend this definition to
the boundary of $\mathbb{R}_{\geq0}^{n}$. For a point $s$ on the
boundary of $\sigma$ we construct a tropical curve $\Gamma_{s}=(G_{s},\ell_{s},w_{s})$
as follows. Let $G_{s}$ be the graph obtained from $G$ by contracting
an edge $e_{k}$ whenever $s_{k}=0$. Define the length function by
$\ell(e_{i})=s_{i}$ whenever $s_{i}>0$, and the weight $w_{s}$
at a vertex $v$ of $G_{s}$ by $w_{s}(v)=g(H_{v})$, where $H_{v}$
is the subgraph of $G$ being collapsed to the point $v$ in $G_{s}$.
For instance, by contracting a loop edge, the weight at the base is
increased by one, and contracting an edge between two vertices identifies
both of them to a single vertex whose weight is the sum of the weights
at both ends.

\begin{remark}
Points on the boundary of $\sigma$ correspond to curves of different
combinatorial types than points in the interior. All points in the
interior of a face of $\sigma$ correspond to tropical curves of the
same type. 
\end{remark}
\begin{example}
If $G$ is the graph shown in Figure \ref{Moduli1},

\begin{figure}[H]
\centering
\includegraphics[scale=0.5]{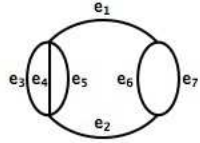}
\caption{The finite graph $G$} \label{Moduli1}
\end{figure}

\noindent and $w$ is the zero function, then the tropical
curves shown in Figure \ref{Moduli2} are all on different faces of $\sigma$.

\begin{figure}[H]
\centering
\includegraphics[scale=0.5]{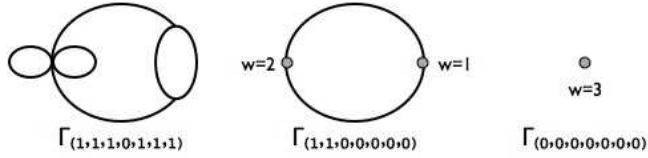}
\caption{Different tropical curves obtained from the graph $G$} \label{Moduli2}
\end{figure}

\end{example}

For every $s=(s_{1},...,s_{n})\in\sigma$ there is a natural rescaling
map $\alpha_{s}:\Gamma_{(1,...,1)}\to\Gamma_{s}$ that induces a map
$H_{1}(\Gamma_{(1,...,1)},\mathbb{Z})\to H_{1}(\Gamma_{s},\mathbb{Z})$
and a surjection of homology groups $H_{1}(\Gamma_{(1,...,1)},\mathbb{R})\twoheadrightarrow H_{1}(\Gamma_{s},\mathbb{R})$,
hence a surjection $\alpha_{s*}:J(\Gamma_{(1,...,1)})\to J(\Gamma_{s})$.
Note that this map is surjective even when edges are contracted. Therefore
we obtain a map 
\[
\alpha:\sigma\times J(\Gamma_{(1,...,1)})\to\underset{s\in\sigma}{\coprod}\{s\}\times J(\Gamma_{s})
\]
 given by 
\[
\alpha(s,D)=(s,\alpha_{s*}(D))
\]

\begin{definition}
Let $(G,w)$ be a combinatorial type of tropical curves.
\end{definition}

\begin{enumerate}
\item The \emph{Universal Jacobian} of the space of tropical
curves of type $(G,w)$ is 
\[
J(G,w)=\underset{s\in\sigma}{\coprod}\{s\}\times J(\Gamma_{s})
\]
endowed with the quotient topology, that is to say the largest topology
making $\alpha$ continuous. 
\item Fix a point $p$ in $\Gamma_{(1,...,1)}$. For every $s$ in $\sigma$,
let $\alpha_{s}(p)$ be its image in $\Gamma_s$ under the natural rescaling
map. By identifying pairs of the form $(s,[D])$ in $\{s\}\times J(\Gamma_{s})$
with pairs of the form $(s,[D-d\cdot\alpha_{s}(p)])$ in $\{s\}\times\mbox{Pic}_{d}(\Gamma_{s})$
we obtain the \emph{Universal Picard Space} 
\[
\mbox{Pic}_{d}(G,w)=\underset{s\in\sigma}{\coprod}\{s\}\times\mbox{Pic}_{d}(\Gamma_{s})\mbox{,}
\]
 with the property that the map $\sigma\times(\Gamma_{(1,\ldots,1)})^{d}\to\mbox{Pic}_{d}(G,w)$,
sending $(s,E)$ to $(s,\alpha_{s*}[E-d\cdot p])$, is continuous.
\item \emph{The universal Brill-Noether locus}, denoted $W_{d}^{r}(G,w)$,
is the subset of $\mbox{Pic}_{d}(G,w)$ consisting of pairs $(s,[D])$
such that the rank of $D$ is at least $r$. 
\end{enumerate}
We can also exhibit the universal Jacobian explicitly as a quotient
of a polyhedral set inside $\mathbb{R}^{2n}$. Let $c_{1},...,c_{g}$
be a basis for the $\mathbb{Z}$-module $H_{1}(\Gamma_{1,...,1},\mathbb{Z})$,
and suppose that $c_{i}=\sum_{j}a_{ij}e_{j}$, where $e_{j}$ are
the edges of $G$ and each $a_{ij}$ is an integer. Notice that for
any choice of $s=(s_{1},...,s_{n})$ in $\mathbb{R}_{\geq0}^{n}$,
the corresponding sums of edges in $\Gamma_{s}$ obtained by the rescaling
map also generate $H_{1}(\Gamma_{s},\mathbb{Z})$. For each $i$,
let $C_{i}$ be the vector of coefficients of $c_{i}$, namely $C_{i}=(a_{i1},a_{i2},...,a_{in})\in\mathbb{R}^{n}$,
and for $(s_{1},...,s_{n})$ in $\sigma$ denote $C_{i}^{s_{1},...,s_{n}}=(s_{1}a_{i1},...,s_{n}a_{in})$.
Define 
\[
\mathcal{H}=\underset{(s_{1},...,s_{n})\in\sigma}{\coprod}\{(s_{1},...,s_{n},\underset{i=1}{\overset{g}{\sum}}t_{i}C_{i}^{s_{1},...,s_{n}})\mid0\leq t_i\leq1)\}\subseteq\sigma\times\mathbb{R}^{n}
\]
(where $t_i$ runs through all the values between $0$ and $1$) and let 
\[
\mathcal{J}=\mathcal{H}/\sim
\]
where $\sim$ is the equivalence relation identifying opposite faces
in the parallelogram above each point in $\sigma$. In other words,
for each $k=1,...,g$, the relation $\sim$ identifies points of the
form 
\[
(s_{1},...,s_{n},\underset{i\neq k}{\overset{}{\sum}}t_{i}C_{i}^{s_{1},...,s_{n}})
\]
 with those of the form 
\[
(s_{1},...,s_{n},\underset{i\neq k}{\overset{}{\sum}}t_{i}C_{i}^{s_{1},...,s_{n}}+C_{k}^{s_{1},...,s_{n}})\mbox{.}
\]
 The restriction of $\mathcal{J}$ to the fiber above any point of
$\sigma$ is naturally isomorphic the Jacobian of the corresponding
tropical curve and we have a natural map 
\[
\psi:\sigma\times J(\Gamma_{(1,...,1)})\to\mathcal{J}
\]
sending 
\[
(s_{1},...,s_{n})\times\overset{g}{\underset{i=1}{\sum}}t_{i}C_{i}^{1,...,1}
\]
to 
\[
\ensuremath{(s_{1},...,s_{n})\times\overset{g}{\underset{i=1}{\sum}}t_{i}C_{i}^{s_{1},...,s_{n}}}\mbox{.}
\]
To show that this definition agrees with the one above, consider the
following commuting diagram:

\[
\xymatrix{ & \sigma\times J(\Gamma_{(1,...,1)})\ar[dl]_{\alpha}\ar[dr]^{\psi}\\
J(G,w)\ar[rr]_{\theta} &  & \mathcal{J}
}
\]
with $\theta$ defined on each fiber $\{s\}\times J(\Gamma_{s})$
of $J(G,w)$ as the one to one and onto map identifying $J(\Gamma_{s})$
with a torus. Let us show that $\theta$ is continuous: suppose $U\subseteq\mathcal{J}$
is an open set. By definition, a set in $J(G,w)$ is open if and only
if its preimage in $\sigma\times J(\Gamma_{(1,...,1)})$ is, hence
$\theta^{-1}(U)$ is open if and only if $\alpha^{-1}(\theta^{-1}(U))$
is. But by the commutativity of the diagram $\alpha^{-1}(\theta^{-1}(U))=\psi^{-1}(U)$
which is open since $\psi$ is continuous. Hence $\theta$ is continuous.
The Jacobian $J(G,w)$ is compact because it is the continuous image
of a compact space and $\mathcal{J}$ is seen from definition to be
Hausdorff. The map $\theta$ is a bijection between a compact and
a Haudorff space, hence a homeomorphism.

\section{\label{sec:The-universal-Brill-Noether}The universal Brill-Noether
locus}

This section is devoted to proving the following proposition:
\begin{theorem}
\label{BN-locus is closed-1}Let $(G,w)$ be a combinatorial type
of tropical curves. Then the universal Brill-Noether locus $W_{d}^{r}(G,w)$
is closed in {\em $\text{Pic}_d(G,w)$\em}. 
\end{theorem}
\noindent \noindent The proposition is equivalent to the fact that,
whenever a sequence $(s_{i},[D_{i}])$ converges to a pair $(s,[D])$
in $\mbox{Pic}_{d}(G,w)$, and for every $i$ the rank of $D_{i}$
is at least $r$, then the limit divisor class $[D]$ has rank at
least $r$ as well. Suppose that $\Gamma_{s}=(G_{s},w_{s},\ell_{s})$
and the vertices of $G_{s}$ are $v_{1},...,v_{m}$. By the discussion
above on rank determining sets, this is the same as showing that for
every effective divisor $E=b_{1}v_{1}+...+b_{m}v_{m}$ of degree $r$, the divisor
$D$ is equivalent to a divisor containing $b_{j}+\min\{b_{j},w(v_{j})\}$
points at every vertex $v_{j}$. Each vertex $v_{j}$ was obtained,
roughly speaking, by collapsing a subgraph $H_{j}$ of genus $w(v_{j})$
of $G$ and that subgraph is realized as a subcurve $\Lambda_{i}$.
Hence we expect that, when $i$ is large enough, the divisor $D_{i}$
will be linearly equivalent to a divisor whose restriction to a neighbourhood
of each subcurve $\Lambda_{i}$ is of degree at least $b_{j}+\min\{b_{j},w(v_{j})\}$.
This leads us to explore the behavior of divisors when the diameter
of certain subcurves is very small with respect to the rest of the
curve. Figure \ref{Sequence} illustrates a sequence $(\Gamma_{k})$
of pure tropical curves converging to a weighted tropical curve $\Gamma$.
The subcurves $\Lambda_{1},\Lambda_{2},\Lambda_{3}$ collapse in the
limit to the weighted vertices $v_{1},v_{2},v_{3}$. 
\begin{figure}[H]
\centering
\includegraphics[scale=0.35]{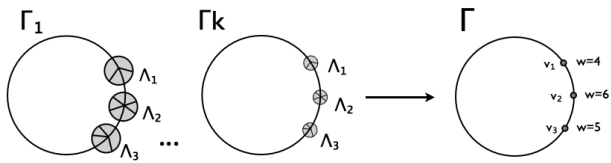}
\caption{A squence of tropical curves converging in $\text{Pic}_d(G,w)$} \label{Sequence}
\end{figure}

See also \cite{GathmannKerber} for a different example, where the limiting behavior of families of
metric graphs was studied in order to extend the Riemann-Roch Theorem to the metric graph
case.  

\begin{lemma}
\label{Chan}Let $D$ and $E$ be two effective linearly equivalent
divisors on a tropical curve $\Gamma$, and let $f$ be a piecewise
linear function such that $D-E=\mbox{div}(f)$. Then at every point
$x$ where $f$ is differentiable we have $|f'(x)|\leq\deg(D)$. 
\end{lemma}
\noindent A similar claim for the case where $\mbox{deg}(D)=2$ appeared
in the second part of \cite[Lemma 3.1]{ChanHyperellipric} and the
proof of the general case is almost identical. We include the details
for completeness.
\begin{proof}
Write $F=\mbox{div}(f)=D-E=x_{1}+\cdots+x_{d}-y_{1}-\cdots-y_{d}$.
Pick a point $z$ in $\Gamma$. Let $s^{+}(z)$ and $s^{-}(z)$ denote
the sum of the positive and negative incoming slopes of $f$ at $z$
respectively. We will show that $s^{+}(z)\leq d$. An analogous argument
shows that $s^{-}(z)\geq-d$, so the lemma will follow. Let $U$ be
the union of all the paths in $\Gamma$ eminating from $z$ along
which $f$ is nondecreasing. Let $z_{1},\ldots,z_{n}$ be the set
of points in $U$ that are either vertices of $\Gamma$ or are points
at which $f$ is non-diff{}erentiable. Let $W=\{z,z_{1},\ldots,z_{n}\}$.
Then $U\backslash W$ consists of fi{}nitely many open segments. Let
$S=\{s_{1},\ldots,s_{k}\}$ be the set of closures of these segments.
So $s_{1},\ldots,s_{k}$ (which could be either closed intervals or loops) cover U and intersect
at points in $W$, and $f$ is linear along each of them. Orient each
segment in $S$ for reference, and for each point $y$ in $W$ let

\[
\delta^{+}(y)=\{j\mbox{ }|\mbox{ }s_{j}\mbox{ is incoming at \ensuremath{y}}\}
\]

\[
\delta^{-}(y)=\{j\mbox{ }|\mbox{ \ensuremath{s_{j}\mbox{ is outgoing at \ensuremath{y\}}}}}
\]
Finally, for each $i=1,\ldots,k$, let $m_{i}$ be the slope of $f$
along the oriented segment $s_{i}$. Note that the slope of $f$ along
any edge $e$ leaving $U$ must be negative, because otherwise $e$
would lie in $U$. Therefore we have

\[
F(z_{i})\geq\sum_{j\in\delta^{+}(z_{i})}m_{j}-\sum_{j\in\delta^{-}(z_{i})}m_{j}
\]

\[
F(z)\geq\sum_{j\in\delta^{+}(z)}m_{j}-\sum_{j\in\delta^{-}(z)}m_{j}+s^{+}(z)
\]
for $i=1,\ldots,n$. Summing over all $i$, we have $F(z)+F(z_{1})+\dots+F(z_{l})\geq s^{+}(z)$.
Since $F=x_{1}+\dots+x_{d}-y_{1}-\dots-y_{d}$ it must be that $F(z)+F(z_{1})+\dots+F(z_{n})$
and hence $s^{+}(z)$ is at most $d$.\end{proof}
\begin{definition}
Let $\Gamma=(G,w,\ell)$ be a tropical curve. A \emph{tropical subcurve}
$\Lambda$ is a triple $(H,w',\ell')$ such that $H$ is a connected subgraph
of $G$, $w'$ is the restriction of $w$ to the vertices of $H$
and $\ell'$ is the restriction of $\ell$ to the edges of $H$. For
a tropical subcurve $\Lambda$, we write $N_{\delta}(\Lambda)$ for
the set of points in $\Gamma$ whose distance from $\Lambda$ is at
most $\delta$.
\end{definition}
We now fix notation for the remainder of the section. Let $\Gamma$
be a tropical curve such that $\Gamma=(G,w,\ell)$ and let $\Lambda$
be a tropical subcurve of diameter $\epsilon>0$ with $\Lambda=(H,w',\ell')$.
$H$ is assumed to be a loopless graph with $m$ vertices. $D$ will
always denote an effective divisor of degree $d$ and rank at least $r$ on
$\Gamma$. 

\begin{proposition}
\label{Pushin single divisor}Let $E$ be an effective divisor of
degree $r$ supported on $\Lambda$. Then there is a closed neighbourhood
$\Lambda'$ of $\Lambda$ which is contained in $N_{d\epsilon}(\Lambda)$
and an effective divisor $D'$ such that $D$ and $D'$ coincide on
$\Lambda$, $D'$ is linearly equivalent to $D$ (as a divisor on
$\Gamma$), and $D'|_{\Lambda'}-E^{*}$ is equivalent (as a divisor
on $\Lambda'$) to an effective divisor.
\end{proposition}

\begin{proof}
Since $r(D)\geq r$, Proposition \ref{internalRank} implies that there is a piecewise linear function $f$ on
$\Gamma$ satisfying $E^{*}\leq D+\mbox{div}(f)$. Denote $\mu=\inf\{{f(x)\mid x\in\partial N_{\epsilon}(\Lambda)}\}$,
and let 
\[
\bar{f}=\begin{cases}
\mu & x\in N_{\epsilon}(\Lambda)\\
\min\{f,\mu\} & x\notin N_{\epsilon}(\Lambda)
\end{cases}\mbox{.}
\]

Define $D'=D+\mbox{div}(\bar{f})$, (see Subsection \ref{prelim} for the definition of $\mbox{div}(\bar{f})$). $D'$ identifies with $D$ on
$\Lambda$ since $\bar{f}$ is constant on a neighbourhood of $\Lambda$,
and one easily checks that $D'$ is effective. Let $\Lambda'$ be
the closure of the union of $\Lambda$ with the paths
emanating from it along which $f$ is decreasing and strictly
larger than $\mu$. By Lemma \ref{Chan}, the absolute value of the
slope of $f$ is at most $d$, hence 
\[
\sup\{f(x)-f(y)\mid x,y\in \Lambda\}\leq d\cdot\epsilon\mbox{,}
\]
 implying that these paths cannot be longer than $d\epsilon$ and
$\Lambda'$ is contained in $N_{d\epsilon}(\Lambda)$. It remains
to show that the restriction of $D'$ to $\Lambda'$ is linearly equivalent
(as a divisor on $\Lambda'$) to an effective divisor containing $E^{*}$.
Denote $g=(f-\bar{f})|_{\Lambda'}$ and let $D''$ be the divisor
$D''=D'|_{\Lambda'}+\mbox{div}(g)$ on $\Lambda'$. $D''$ is equivalent
on $\Lambda'$ to $D'|_{\Lambda'}$ and the claim is that it is effective
and contains $E^{*}$. Since this is true for $(D+\mbox{div}(f))|_{\Lambda'}$,
it suffices to check that $D''\geq(D+\mbox{div}(f))|_{\Lambda'}$. 

Let $y$ be any point of $\Lambda'$. If $y$ is in the interior of
$\Lambda'$ then $g$ identifies with $(f-\bar{f})$ on a neighbourhood
of $y$, hence $\mbox{div}(g)(y)=\mbox{div}(f-\bar{f})(y)$ and $D''(y)=(D+\mbox{div}(f))(y)$.
If $y\in\partial\Lambda'$ and $f(y)>\mu$ then on each path emanating
from $\Lambda'$ at $y$, $f$ is non-decreasing (otherwise, by the construction
of $\Lambda'$ that path would have been included in $\Lambda'$).
Hence $\mbox{div}(g)(y)\geq\mbox{div}(f-\bar{f})(y)$ and $D''(y)\geq(D+\mbox{div}(f))(y)$.
Finally, suppose that $y\in\partial\Lambda'$ and $f(y)=\mu$. Let
$e$ be a segment leading to $y$ from outside $\Lambda'$ along which
$f$ has constant slope. If the incoming slope of $f$ is nonnegative
then $\bar{f}$ identifies with $f$ on $e$ and $(f-\bar{f})=0$
on $e$. If the incoming slope of $f$ is negative then the incoming
slope of $f-\bar{f}$ is negative as well. In either case, $\mbox{div}(g)(y)\geq\mbox{div}(f-\bar{f})(y)$
and $D''(y)\geq D+(\mbox{div}(f))(y)$.\end{proof}

\begin{proposition}
\label{Pushing} Suppose that the\textup{ closed neighbourhood $\mbox{ }N_{\epsilon\cdot (3d)^{d-r+1}}(\Lambda)$} deformation retracts
onto $\Lambda$. Then there is a neighbourhood $\widetilde{\Lambda}$
of $\Lambda$ which is contained in $N_{\epsilon\cdot (3d)^{d-r+1}}(\Lambda)$
and an effective divisor $\widetilde{D}$ such that:
\begin{enumerate}
\item $D$ is equivalent to $\widetilde{D}$ as a divisor on $\Gamma$.
\item $D$ and $\widetilde{D}$ coincide on $\Lambda$.
\item The rank of $\widetilde{D}$ as a divisor on $\widetilde{\Lambda}$
is at least $r$. \end{enumerate}
\end{proposition}
\begin{proof}
Recall that by Corollary \ref{prop:weighted RDS}, the vertices
of $\Lambda$ are a rank determining set for $\Lambda$. Since $N_{\epsilon\cdot (3d)^{d-r+1}}(\Lambda)$
deformation retracts onto $\Lambda$ and does not contain any weighted
vertices outside $\Lambda$, they are a rank
determining set for $N_{\epsilon\cdot (3d)^{d-r+1}}(\Lambda)$ as well.
Let $S_{1},\ldots,S_{M}$ be the different effective divisors of degree
$r$ supported on the vertices of $\Lambda$. 
By Proposition (\ref{Pushin single divisor}) with $\Lambda$ and
$S_{1}$ we obtain a curve $\Lambda_{1}$ of diameter $\epsilon_{1}\leq3d\cdot\epsilon$
and a divisor $D_{1}$, such that $D_{1}$ is linearly equivalent
to $D$, identifies with it on $\Lambda$ and $D_{1}|_{\Lambda_{1}}-S_{1}^{*}$
is linearly equivalent, as a divisor on $\Lambda_{1}$, to an effective
divisor. Note that the degree of $D_{1}|_{\Lambda_{1}}$ is at least $r$. 

We now repeat the process successively for $S_{2},...,S_{M}$ as follows. For $2\leq i\leq M$, we use Proposition (\ref{Pushin single divisor}) again to obtain a subcurve $\Lambda_i$ and a divisor $D_i$, such that the restriction
of $D_{i}$ to $\Lambda_{i-1}$ identifies with $D_{i-1}$, and for every $j=1,\ldots,i$, $D_{i}|_{\Lambda_{i}}-S_{j}^{*}$
is equivalent, as a divisor on $\Lambda_{i}$, to an effective divisor.

If we denote $\epsilon_i$ the diameter of $\Lambda_i$, then by construction, $\epsilon_i < 3d\cdot\epsilon_{i-1}$. Now, if $\rm{deg}(D_i|_{\Lambda_i})$ is the same as $\rm{deg}(D_{i-1}|_{\Lambda_i})$, then $D_{i-1}$ is already equivalent (as a divisor on $\Lambda_{i-1}$) to a divisor that contains $S_i$. Therefore, we may choose $\Lambda_{i}$ to be the same subcurve as $\Lambda_{i-1}$, and in this case, $\epsilon_i=\epsilon_{i-1}$. In particular, since the total degree of $D$ is $d$, there can be at most $d-r+1$ stages where the degree of the restriction increases, and so $\epsilon_M < \epsilon\cdot(3d)^{d-r+1}$.

Now choose $\widetilde{D}=D_{M}$ and $\widetilde{\Lambda}=\Lambda_{M}$.
By construction, $\widetilde{D}$ is linearly equivalent to $D$, and for every subset
$S$ of degree $r$ of the rank determining set, $\widetilde{D}|_{\widetilde{\Lambda}}$
is equivalent on $\widetilde{\Lambda}$ to a divisor containing $S^{*}$.
It follows that $r_{\widetilde{\Lambda}}(\widetilde{D}|_{\widetilde{\Lambda}})\geq r$. \end{proof}
\begin{definition}
Let $\Lambda$ be a subcurve of a tropical curve $\Gamma$ and let
$r\in\mathbb{Z}_{\geq0}$. Then we write $r^{\Lambda}=r+\min\{r,g(\Lambda)\}$.
\end{definition}
\begin{corollary}
\label{Pushin cor}In the same conditions and notations as Proposition
\ref{Pushing}, $D$ is linearly equivalent to an effective divisor whose restriction
to $\widetilde{\Lambda}$ has degree at least $r^{\Lambda}$. Moreover,
this divisor coincides with $D$ on $\Lambda$.
\end{corollary}
\begin{proof}
Let $\widetilde{D}$ be the divisor linearly equivalent to $D$ obtained
in Prop (\ref{Pushing}). Then $\widetilde{D}$ identifies with $D$
on $\Lambda$ and $r_{\widetilde{\Lambda}}(\widetilde{D}|_{\widetilde{\Lambda}})\geq r$.
Let $E$ be the restriction of $\widetilde{D}$ to $\widetilde{\Lambda}$,
denote $e=\deg E$ and let $h$ be the genus of $\widetilde{\Lambda}$.
We must show that $e\geq r^{\Lambda}$.

If $e\geq2h$ then by the Riemann-Roch theorem for tropical curves
(\cite[Theorem 5.4]{CaporasoAmini}), $r_{\tilde{\Lambda}}(E)=e-h$, hence $e=r_{\tilde{\Lambda}}(E)+h\geq r+ h \geq r+\min\{r,h\}$. 

Now suppose $0\leq e<2h$.
$E$ could either be a special (i.e. $r_{\tilde{\Lambda}}(E)>e-h$) or a non-special divisor,
and we claim that in any case $2\cdot r_{\tilde{\Lambda}}(E)\leq e$: if $E$ is special this is exactly
Clifford's theorem (\cite[Theorem 1]{Facchini}). Otherwise,
$E$ is non-special, namely $r_{\tilde{\Lambda}}(E)=e-h$. Assume by contradiction that
$e<2\cdot r_{\tilde{\Lambda}}(E)=2e-2h$. But then $2h<e$, a contradiction. In
any case, $2r\leq 2\cdot r_{\tilde{\Lambda}}(E)\leq e$. Together with the fact that $e<2h$, it follows
that $r<h$, hence $e\geq2r=r+\min\{r,h\}=r^{\Lambda}$. 

\end{proof}
Suppose that instead of a single subcurve we are given two subcurves
$\Lambda_{1}$ and $\Lambda_{2}$ and two integers $r_{1}$ and $r_{2}$
such that $r_{1}+r_{2}=r$. One could be tempted to use Corollary
\ref{Pushin cor} twice to show that $D$ is equivalent to an effective
divisor whose restriction to both $\widetilde{\Lambda}_{1}$ and $\widetilde{\Lambda}_{2}$
has degree at least $r_{1}^{\Lambda_{1}}$ and $r_{1}^{\Lambda_{2}}$
respectively. However, when using the corollary a second time in an
attempt to move points of the divisor into $\widetilde{\Lambda}_{2}$,
there is no guarantee, a priori, that enough of its points will remain
inside $\widetilde{\Lambda}_{1}$. The following lemma shows that
every subcurve contains special configurations of points which are
``stuck'' inside the subcurve. That is to say, whenever the restriction
of an effective divisor to the subcurve is exactly that configuration,
it cannot be moved to contain fewer points in the subcurve.
\begin{lemma}
\label{Confinement}Let $h$ be the genus of the underlying metric
space $\Lambda^{0}$ and let $0\leq k\leq h$. Then there exists a
divisor $v_{1}+\cdots+v_{k}$ of degree $k$ supported on $\Lambda$
such that whenever $E$ and $E'$ are two linearly equivalent effective
divisors, and the restriction of $E$ to $\Lambda$ is exactly $v_{1}+\cdots+v_{k}$,
then the degree of the restriction of $E'$ to $\Lambda$ is at least
$k$. \end{lemma}
\begin{proof}
First assume that $k=h$. We will show that the set of classes of
effective divisors on $\Lambda$ that do not satisfy the lemma cannot
be the entire $h\mbox{-dimensional}$ torus $\mbox{Pic}_{h}(\Lambda)$.
Suppose that $E_{\Lambda}$ is such a divisor. Namely, there is an
effective divisor $E$ that restricts to $E_{\Lambda}$ on $\Lambda$,
and $E$ is linearly equivalent to an effective divisor $E'$ whose
restriction to $\Lambda$ has degree less than $h$. Write 
\[
E=E'+\mbox{div}(f)
\]
for some piecewise linear function $f$. The divisors $\mbox{div}(f)|_{\Lambda}$
and $\mbox{div}(f|_{\Lambda})$ coincide on the interior of $\Lambda$
hence differ by a combination of its boundary points. Therefore, if
we denote the boundary points of $\Lambda$ by $y_{1},...,y_{\ell}$,
then 
\[
E|_{\Lambda}=E'|_{\Lambda}+\mbox{div}(f|_{\Lambda})+a_{1}\cdot y_{1}+\cdots+a_{\ell}\cdot y_{\ell}
\]
for some integers $a_{1},...,a_{\ell}$. Hence the divisor $E|_{\Lambda}$
is linearly equivalent, as a divisor on $\Lambda$, to $E'|_{\Lambda}+a_{1}\cdot y_{1}+\cdots a_{\ell}\cdot y_{\ell}$.
For a fixed $a_{1}y_{1}+\cdots+a_{\ell}y_{\ell}$, every choice of
$E|_{\Lambda}$ corresponds to a choice of $E'|_{\Lambda}$ which
is a divisor of degree strictly smaller than $h$. It follows that
the set of all such $[E|_{\Lambda}]$ is the translation by $ $$a_{1}y_{1}+\cdots+a_{\ell}y_{\ell}$
of the set of effective divisors of a fixed degree smaller than $h$.
Recall that this set has positive codimension inside $\mbox{Pic}_{h}(\Lambda)$
(\cite[Proposition 3.6]{LPP}). As we let $a_{1}y_{1}+\cdots+a_{\ell}y_{\ell}$
vary over all the divisors supported on the boundary we see that the
set of classes $[E_{\Lambda}]$ that do not satisfy the lemma is a
countable union of sets of positive codimension, and hence is a proper
subset of $\mbox{Pic}_{h}(\Lambda)$. 

Now, let $E$ be an effective divisor whose restriction to $\Lambda$
is $v_{1}+...+v_{k}$ for some $k<h$ and let $E'$ be an effective
divisor such that $E'=E+\mbox{div}(f)$ for some piecewise linear
function $f$. Then $E'=(E+v_{k+1}+...+v_{h}+\mbox{div}(f))-v_{k+1}-...-v_{h}$.
By the first part, $\deg((E+v_{k+1}+...+v_{h}+\mbox{div}(f))|_{\Lambda})\geq h$.
Therefore $\deg(E'|_{\Lambda})\geq h-(h-k)=k$. 
\end{proof}
Since the lemma above requires the restriction of the divisor to $\Lambda$
to be exactly $v_{1}+\cdots+v_{k}$, we shall require a tool to remove
superfluous points away from the subcurve. This is given by the following
lemma:
\begin{lemma}
\label{Diluting}Let $E$ be an effective divisor on $\Gamma$. Suppose
that the restriction of $E$ to $\Lambda$ has degree greater than
$k$ but $E$ is equivalent to an effective divisor $F$ whose restriction
to $\Lambda$ has degree smaller than $k$. Then there exists an arbitrarily
small neighbourhood $\Lambda'$ of $\Lambda$ and an effective divisor
$E'$ linearly equivalent to $E$ such that its restriction to $\Lambda'$
has degree exactly $k$. 
\end{lemma}
\begin{proof}
Let $f$ be the piecewise linear function satisfying $E+\mbox{div}(f)=F$.
Let $\mbox{outdeg}_{\Lambda}(f)$ denote the sum of all the slopes
of $f$ emanating from $\Lambda$. Then \linebreak{}
$\mbox{outdeg}_{\Lambda}(f)=\deg(E|_{\Lambda})-\deg(F|_{\Lambda})>0$.
For each path eminating from $\Lambda$, choose a segment homeomorphic
to an interval along which $f$ has a constant slope and contains
no points of $E$ except possibly at their intersection with $\Lambda$.
Let $\{e_{1},...,e_{l}\}$ be the sets of those paths. Note that each
$e_{i}$ can be chosen arbitrarily small. For each $e_{i}$, let $p_{i}$
be its endpoint touching $\Lambda$ and $q_{i}$ the other endpoint,
and denote $c_{i}$ the outgoing slope of $f$ on $e_{i}$. Rearrange
the segments so that $i_{2}>i_{1}$ whenever $f(p_{i_{2}})>f(p_{i_{1}})$
or $f(p_{i_{2}})=f(p_{i_{1}})$ and $f(q_{i_{2}})>f(q_{i_{1}})$.
Let $\alpha$ be the first index satisfying

\[
\deg(E|_{\Lambda})-\overset{\alpha}{\underset{i=1}{\sum}}c_{i}<k\mbox{.}
\]

Let $\Lambda'$ be the union of $\Lambda$ with the segments $e_{\alpha+1},...,e_{l}$.
Denote $c=\deg(E|_{\Lambda})-\overset{\alpha-1}{(\underset{i=1}{\sum}}c_{i}+k)>0$
and let $M=f(p_{\alpha})+\epsilon\cdot c$ for some positive number
$\epsilon$ smaller than the length of $e_{\alpha}$. Now define

\[
\bar{f}(x)=\begin{cases}
f(p_{\alpha})+c\cdot x & x\in e_{\alpha},0\leq x\leq\epsilon\\
\max(f(x),M) & \mbox{otherwise}
\end{cases}\mbox{,}
\]
and let $E':=E+\mbox{div}(\bar{f})$. Then 
\[
\mbox{outdeg}_{\Lambda'}(\bar{f})=\overset{}{\overset{\alpha-1}{\underset{i=1}{\sum}}c_{i}+c=\overset{\alpha-1}{\underset{i=1}{\sum}}}c_{i}+\deg(E|_{\Lambda})-(k+\overset{\alpha-1}{\underset{i=1}{\sum}}c_{i})=\deg(E|_{\Lambda})-k\mbox{,}
\]
 hence $\deg(E'|_{\Lambda'})=\deg(E|_{\Lambda})-\mbox{outdeg}_{\Lambda'}(\bar{f})=k$
as required. 
\end{proof}
The next proposition is a stronger version of Corollary \ref{Pushin cor}.
We use it in Proposition \ref{Arrangement multiple graphs} to generalize
the result of the corollary to multiple tropical subcurves. 
\begin{proposition}
\label{Arrangement single graph}Suppose that the neighbourhood $ $$\mbox{ }N_{\epsilon\cdot (3d)^{d-r+1}}(\Lambda)$
deformation retracts onto $\Lambda$ and does not contain any weighted
vertices outside $\Lambda$. Then there is an effective divisor $\widetilde{D}$
which is linearly equivalent to $D$, a neighbourhood $\widetilde{\Lambda}$
of $\Lambda$ which is contained in $N_{\epsilon\cdot (3d)^{d-r+1}}(\Lambda)$,
and an effective divisor $U$ of degree $r$ supported on $\widetilde{\Lambda}$,
such that:
\begin{enumerate}
\item $\tilde{D}$ contains $U$.
\item Any divisor that contains $U$ and is linearly equivalent to $\widetilde{D}$,
restricts to a divisor of degree at least $r^{\Lambda}$ on $\tilde{\Lambda}$. 
\end{enumerate}
\end{proposition}
\begin{proof}
Let $s=\min\{r,h\}$ where $h$ is the genus of $\Lambda^{0}$. Let
$V$ be the effective divisor of degree $s$ supported on $\Lambda$
that was obtained in Lemma \ref{Confinement}. Since $r(D)\geq r$,
we can replace $D$ with a linearly equivalent effective divisor that
contains $V$. By Corollary \ref{Pushin cor}, there is a divisor
$D'$ which is equivalent to $D$, identifies with it on $\Lambda$,
and satisfies $\deg(D'|_{\Lambda'})\geq r^{\Lambda}$ for some neighbourhood
$\Lambda'$ of $\Lambda$ in $N_{\epsilon\cdot (3d)^{d-r+1}}(\Lambda)$.
In particular, $D'$ contains $V$, and $D'-V$ has degree at least $r$ on $\Lambda'$. 

Suppose that $D'-V$ is equivalent to some divisor whose restriction
to $\Lambda'$ has degree less than $r$. Then by Lemma \ref{Diluting},
$D'-V$ is equivalent to an effective divisor $D''$ whose restriction
to an arbitrarily small neighbourhood $\tilde{\Lambda}$ of $\Lambda'$
has degree exactly $r$. Choose $\tilde{\Lambda}$ small enough so
that it is still contained in $N_{\epsilon\cdot (3d)^{d-r+1}}(\Lambda)$.
Let $U$ be the restriction of $D''$ to $\tilde{\Lambda}$. Take
$\tilde{D}=D''+V$. Then $\tilde{D}$ is equivalent to $D'$ which
is equivalent to $D$. Notice that the restriction of $\widetilde{D}$
to $\widetilde{\Lambda}$ is exactly $V+U$, hence the restriction
of $\widetilde{D}-U$ is exactly $V$. Now, let $E$ be any other
effective divisor that contains $U$ and is linearly equivalent to
$\widetilde{D}$. Then $E-U$ is effective and linearly equivalent
to $\widetilde{D}-U$, so by the choice of $V$, the restriction of
$E-U$ to $\widetilde{\Lambda}$ has degree at least $s=r^{\Lambda}-r$.
Therefore, $\deg(E|_{\tilde{\Lambda}})\geq s+r=r^{\Lambda}$. 

Otherwise, the restriction to $\Lambda'$ of every effective divisor
that is linearly equivalent to $D'-V$ has degree at least $r$. In
this case, the restriction to $\Lambda'$ of every divisor that is
linearly equivalent to $D'$ and contains $V$ has degree at least
$r+s=r^{\Lambda}$. Hence the proposition is satisfied by choosing
$\widetilde{D}=D'$, $\widetilde{\Lambda}=\Lambda'$, and $U$ to be
any effective divisor of degree $r$, which contains $V$, contained in $\widetilde{D}$, and is supported on $\Lambda$.
\end{proof}
The following is the generalization of Corollary \ref{Pushin cor}
to multiple subcurves. Suppose that $r_{1},\ldots,r_{k}$ are non-negative
numbers such that $r_{1}+\ldots+r_{k}\leq r$, and $\Lambda_{1},\ldots,\Lambda_{k}$
are tropical subcurves of $\Gamma$ such that $\Lambda_{i}=(H_{i},w_{i},l_{i})$,
each $H_{i}$ is loopless, and the diameter
of each $\Lambda_{i}$ is $\epsilon_{i}$.

\begin{proposition}
\label{Arrangement multiple graphs} Suppose that for every $i$,
the curve $N_{\epsilon_{i}\cdot (3d)^{d-r_i+1}}(\Lambda_{i})$
deformation retracts onto $\Lambda_{i}$ and contains no weighted
vertices outside $\Lambda_{i}$. Then there is a neighbourhood $\widetilde{\Lambda}_{i}$
of each $\Lambda_{i}$ which is contained in $N_{\epsilon_{i}\cdot (3d)^{d-r_i+1}}(\Lambda_{i})$,
and an effective divisor linearly equivalent to $D$, whose restriction to each
$\widetilde{\Lambda}_{i}$ is of degree at least $r_{i}^{\Lambda_{i}}$. 
\end{proposition}
\begin{proof}
We will construct, using induction on $k$, an effective divisor $D_{k}$,
a seqeunce of effective divisors $U_{1},\ldots,U_{k}$, and neighborhoods
$\widetilde{\Lambda}_{1},\ldots,\widetilde{\Lambda}_{k}$ with the
following properties: for every $i$, $\widetilde{\Lambda}_{i}$ is
contained in $N_{\epsilon_{i}\cdot (3d)^{d-r_i+1}}(\Lambda_{i})$
and contains $\Lambda_{i}$; $U_{i}$ is of degree $r_{i}$ and is
supported on $\widetilde{\Lambda}_{i}$; the restriction to $\widetilde{\Lambda}_{i}$
of any divisor that is linearly equivalent to $D$ and contains $U_{i}$
is of degree at least $r_{i}^{\Lambda_{i}}$; $D_{k}$ is linearly
equivalent to $D$ and contains $U_{1},\ldots,U_{k}$. Then $D_{k}$
satisfies the statement of the proposition.

For $k=1$, this is just Proposition \ref{Arrangement single graph}.
Now suppose that $k$ is at least $2$ and that the claim is true
for $i=k-1$. Then $D_{k-1}$ contains $U_{1},...,U_{k-1}$, hence
$D_{k-1}-U_{1}-\cdots-U_{k-1}$ is effective and of rank at least $r-r_{1}-\cdots-r_{k-1}\geq r_{k}$.
Moreover, $D_{k-1}$ is linearly equivalent to $D$, so we can use
Proposition \ref{Arrangement single graph} again for the the divisor
$D_{k-1}-U_{1}-...-U_{k-1}$ and the subcurve $\Lambda_{k}$. We obtain
an effective divisor $U_{k}$ of degree $r_{k}$, an effective divisor
$D_{k}'$ that contains $U_{k}$, and a neighbourhood $\widetilde{\Lambda}_{k}$
of $\Lambda_{k}$ such that $D_{k}'$ is linearly equivalent to $D_{k-1}-U_{1}-\cdots-U_{k-1}$
and every effective divisor which is linearly equivalent to $D_{k}'$
and contains $U_{k}$ has degree at least $r_{k}^{\Lambda_{k}}$ on
$\widetilde{\Lambda}_{k}$. Now take $D_{k}=D_{k}'+U_{1}+\cdots+U_{k-1}$
and the statement of the induction is satisfied. 
\end{proof}
We are now in a position to prove that the Brill-Noether
locus is closed in the universal Jacobian: 
\begin{proof}[of Theorem \ref{BN-locus is closed-1}]
 Let $(s_{i},[D_{i}])$ be a sequence in $\mbox{Pic}_{h}(G,w)$ converging
to $(s,[D])$, and suppose that for every $i$, $D_{i}$ is a divisor
of degree $d$ and rank at least $r$ on $\Gamma_{s_{i}}$. By passing
to a subsequence, we may assume that the points $s_{i}$ are all in
the interior of the same face $\sigma$ of $\mathbb{R}_{\geq0}^{|E(G)|}$.
By replacing $(G,w)$ with the combinatorial type corresponding to
$\sigma$, we can assume that every curve $\Gamma_{s_{i}}$ is of
type $(G,w)$ and that $s$ is a point of $\sigma$, possibly on the
boundary. 

Recall the map 
\[
\Psi^{d}:\sigma\times(\Gamma_{(1,...,1)})^{d}\to\mbox{Pic}_{h}(G,w)
\]
defined in section \ref{Section: Universal Jacobian}, and let $\beta,\alpha_{i}$
be the natural rescaling maps 

\[
\xymatrix{ & \Gamma_{(1,...,1)}\ar[dl]_{\alpha_{i}}\ar[dr]^{\beta}\\
\Gamma_{s_{i}} &  & \Gamma_{s}
}
\]
(note that every $\alpha_{i}$ is a homeomorphisms while \textbf{$\beta$}
is a retract). Let $v_{1},\ldots,v_{k}$ be the vertices of $\Gamma_{s}$
which are either weighted or whose preimage under $\beta$ is more
than a single point. Let $\Lambda_{1},\ldots,\Lambda_{k}$ be the
disjoint preimages of $v_{1},\ldots,v_{k}$ in $\Gamma_{(1,...,1)}$.
The weight at every point $x$ of $\Gamma_{s}$ equals the genus of
$\beta^{-1}(x)$, hence the preimage of any point other than $v_{1},\ldots,v_{k}$
is a single non weighted point. For every $1\leq j\leq k$, and every $i$, let
$\epsilon_{i}^{j}$ be the diameter of $\alpha_{i}(\Lambda_{j})$.
Fix a rank determining set for $\Gamma_{s}$ that contains $v_{1},...,v_{k}$,
namely $A=\{v_{1},\ldots,v_{k},v_{k+1},\ldots,v_{K}\}$ for some points
$v_{k+1},\ldots,v_{K}$.

Now, let $E=r_{1}v_{1}+\cdots+r_{k}v_{k}+r_{k+1}v_{k+1}+\cdots+r_{K}v_{K}$
be an effective divisor of degree $r$ on $\Gamma_{s}$. We need to
show that $D$ is linearly equivalent to an effective divisor containing
$E^{*}=r_{1}^{\Lambda_{1}}v_{1}+\cdots+r_{k}^{\Lambda_{k}}v_{k}+r_{k+1}v_{k+1}+\cdots+r_{K}v_{K}$.
Since for every $j$, the subcurve $\alpha_{i}(\Lambda_{j})$ collapses
to a point in the limit $\beta(\Lambda_{j})$, we can find $i$ large
enough so that each of the neighbourhoods $\mbox{ }N_{\epsilon_{i}^{j}(3d)^{d-r_j+1}}(\alpha_{i}(\Lambda_{j}))$
deformation retracts onto $\alpha_{i}(\Lambda_{j})$. Since $D_i - (r_{k+1}v_{k+1}+\cdots+r_{K}v_{K})$ has rank at least $r_1 + \cdots + r_k$, Proposition
\ref{Arrangement multiple graphs} implies that there are closed neighborhoods
$\Delta_{i}^{j}$ of each curve $\alpha_{i}(\Lambda_{j})$ such that,
by replacing $D_{i}$ with a linearly equivalent divisor, we can assume
that it contains $r_{k+1}v_{k+1}+\cdots+r_{K}v_{K}$ and its restriction
to $\Delta_{i}^{j}$ has degree at least $r_{j}^{\Lambda_{j}}$. 

$ $Let $\hat{D}_{i}$ be$ $ the preimage of $D_{i}$ in $(\Gamma_{(1,...,1)})^{d}$.
The fact that $\alpha_{i}$ is a homeomorphism implies that the restriction
of $\hat{D}_{i}$ to $\alpha_{i}^{-1}(\Delta_{i}^{j})$ has degree
$r_{j}^{\Lambda_{j}}$ as well.$ $ Hence, for every $j$ we may choose
an effective divisor $C_{i}^{j}$ of degree $r_{j}^{\Lambda_{j}}$
on $\Gamma_{(1,...,1)}$ that is contained in the restriction of $\hat{D}_{i}$
to $\alpha_{i}^{-1}(\Delta_{i}^{j})$. By compactness we may assume
that the sequence $\hat{D}_{i}$ converges in $(\Gamma_{(1,...,1)})^{d}$
to some $\hat{D}$, and similarly that each sequence $C_{i}^{j}$
converges to some $C^{j}$ such that $C^{j}\leq\hat{D}$. Each $C_{i}^{j}$
is supported on $\alpha_{i}^{-1}(\Delta_{i}^{j})$, so $C^{j}$ is
supported on $\underset{i=1}{\overset{\infty}{\cap}}\alpha_{i}^{-1}(\Delta_{i}^{j})$.
Since the diameters of $\Delta_{i}^{j}$ tend to zero as $i$ tends
to infinity, the set $\underset{i=1}{\overset{\infty}{\cap}}\alpha_{i}^{-1}(\Delta_{i}^{j})$
in $\Gamma_{(1,...,1)}$ is exactly $\Lambda_{i}^{j}$. Therefore
$C^{j}$ is supported on $\Lambda_{i}^{j}$ and the degree of the
restriction of $\hat{D}$ to $\Lambda_{j}^{i}$ is at least $r_{j}^{\Lambda_{j}}$
for every $j$. By continuity of $\Psi^{d}$, the pair $(s,\hat{D})$
is mapped to $(s,[D])$ in $\mbox{Pic}_{h}(G,w)$, hence $D$ is equivalent
to a divisor that contains $E^{*}$. 

This is true for every effective divisor $E$ of degree $r$ supported
on the rank determining set $A$, and it follows that $r(D)\geq r$.
\end{proof}

\section{\label{Section: BN rank }The Brill-Noether rank of a tropical curve}

In the classical case, the dimension of the Brill-Noether locus is
an important invariant of an algebraic curve. Among its different
properties, the function that takes a curve $C$ to $\dim(W_{d}^{r}(C))$
is upper semicontinuous on the moduli space of algebraic curves of
a fixed genus. However, this is no longer true for metric graphs (\cite[Example 2.1]{LPP}).
As a substitute for the dimension of the locus, Lim, Payne and Potashnik
introduced the notion of Brill-Noether rank:
\begin{definition}
(\cite{LPP})Let $\text{\ensuremath{\Gamma}}=(G,\ell)$ be a metric
graph and let $r,d$ be non-negative integers. The \emph{Brill-Noether
rank} $w_{d}^{r}(\Gamma)$ is the largest number $\rho$ such that
every effective divisor of degree $r+\rho$ is contained in an effective
divisor $D$ of degree $d$ and rank at least $r$.
\end{definition}
The authors showed that the function $w_{d}^{r}$ is in fact upper
semicontinuous on the space classifying metric graphs of a given genus
(\cite[Theorem 1.6]{LPP}). For the case of a tropical curve we propose
the following defintion:
\begin{definition}
Let $\Gamma=(G,w,\ell)$ be a tropical curve and let $r,d$ be non-negative
integers. The \emph{Brill-Noether rank$ $} $w_{d}^{r}(\Gamma)$
is the largest number $\rho$ such that every effective divisor $E$
of degree $r+\rho$ is contained in an effective divisor $D$ of degree
$d$ and rank at least $r$.
\end{definition}
Recall that, by definition, this means that $E$ and $D$ are divisors
on $\Gamma^{0}$, but the rank of $D$ is computed as a divisor on
the curve $\Gamma_{\epsilon}^{w}$ for any $\epsilon>0$. We will
show in  \ref{BN rank is USC} that with this definition, the
function assigning a tropical curve its Brill-Noether rank is upper
semicontinuous on $M_{g}^{\text{tr}}$, the moduli space of tropical
curve. We briefly review the construction of $M_{g}^{\text{{tr}}}$,
following \cite{Viviani}. 

Fix a combinatorial type $(G,w)$. Recall that we may identify points
in $\mathbb{R}_{>0}^{E(G)}$ with curves of type $(G,w)$, by assigning
lengths to the edges of $G$ according to the coordinates of the points.
We may then identify points on the boundary of $\mathbb{R}_{\geq0}^{E(G)}$
with curves in which the appropriate edges are contracted. To get
a unique representation we should identify curves obtained from weight
preserving symmetries. Let $\mbox{Aut}(G,w)$ be the group of weight
preserving automorphisms of $G$. Then $\mbox{Aut}(G,w)$ acts on
$\mathbb{R}_{\geq0}^{E(G)}$ by simply permutating coordinates. Denote
\[
\overline{C(G,w)}=\mathbb{R}_{\geq0}^{E(G)}/\mbox{Aut}(G,w)
\]
the closure of the cell classifying all curves of type $(G,w)$. Consider
$\coprod\overline{C(G,w)}$$ $, the disjoint union of all the cells,
and let $\sim$ be the equivalence relation identifying two points
if they correspond to isomorphic curves. 
\begin{definition}
The moduli space of tropical curve of genus $g$ is 
\[
M_{g}^{\text{tr}}=\coprod\overline{C(G,w)}/\sim\mbox{.}
\]
\end{definition}
\begin{theorem}
\label{BN rank is USC}The Brill-Noether rank is upper semicontinuous
on the moduli space of tropical curves of genus $g$.
\end{theorem}
\begin{proof}
It suffices to show that the Brill-Noether rank is upper semicontinuous
on the closure of each of the cells comprising $M_{g}^{\text{tr}}$.
Fix a combinatorial type $(G,w)$, and let $\sigma=\mathbb{R}_{\geq0}^{|E(G)|}$.
Let $[s_{i}]$ be a sequence in $\overline{(G,w)}$ such that the
Brill-Noether rank of each $\Gamma_{s_{i}}$ is at least $\rho$,
and suppose that $[s_{i}]$ converges to $[s]$. We may choose representatives
$s_{i}$ so that they converge to $s$ in $\sigma$. Now, we need
to show that $w_{d}^{r}(\Gamma_{s})\geq\rho$.

Let $E$ be an effective divisor of degree $r+\rho$ on $\Gamma_{s}$.
For every $i$, choose an effective divisor $E_{i}$ of degree $r+\rho$
on $\Gamma_{s_{i}}$, so that the sequence $(s_{i},[E_{i}])$ converges
to $(s,[E])$ in $\mbox{Pic}_{r+\rho}(G,w)$. Since each of the curves
$\Gamma_{s_{i}}$ has Brill-Noether rank at least $\rho$, there exists
a divisor $D_{i}$ containing $E_{i}$, such that each $[D_{i}]$
is in $W_{d}^{r}(G,w)$. By passing to a subsequence we may assume
that the sequence $(s_{i},[D_{i}])$ converges in $\mbox{Pic}_{d}(G,w)$
to some $(s,[D])$ such that $E$ is contained in $D$. By Theorem
\ref{BN-locus is closed-1}, $W_{d}^{r}(G,w)$ is closed in $\mbox{Pic}_{d}(G,w)$,
so $(s,[D])$ is in $W_{d}^{r}(G,w)$ as well, and the rank of $D$
is at least $r$.
\end{proof}
To conclude the paper we wish to examine the relation between the
Brill-Noether rank of a tropical curve and the dimension of the Brill-Noether
locus of an algebraic curve.

Let $X_{0}$ be a nodal curve. The \emph{dual graph} of $X_{0}$
is the metric graph $(G,\ell)$, where the vertices of $G$ are identified
with the irreducible components of $X_{0}$, and an edge of length
$1$ connects two vertices whenever the two corresponding components
meet. In particular, loop edges correspond to self intersection of
components of the special fiber. 

As defined in \cite{CaporasoAmini}, the \emph{dual weighted graph}
of $X_{0}$ is the weighted metric graph $(G,w,\ell)$, where $(G,\ell)$
is the dual graph of $X_{0}$, and the weight at a vertex $v$ is
the geometric genus of the corresponding component. Note that the
genus of the dual weighted graph is the same as the genus of $X_{0}$.

In \cite{BakerSpecialization} the author presented a specialization
lemma relating the rank of divisors on algebraic curves with the rank
of divisors on their dual graphs. This lemma was generalized in \cite[Theorem 4.9]{CaporasoAmini}
and we will need the following special case:
\begin{theorem}
(\cite[Theorem 1.2]{AminiBaker})Let $X$ be a smooth projective curve
over a discretely valued field with a regular semistable model $\mathcal{X}$
whose special fiber $X_{0}$ has dual weighted graph $\Gamma=(G,w,\ell)$.
Then there is a degree preserving specialization map $\tau:\mbox{Div}(X)\to\mbox{Div}(G)$
such that for every $D\in\mbox{Pic}(X)$, $r_{X}(D)\leq r_{\Gamma}(\tau(D))$.
\end{theorem}
In \cite[Theorem 1.7]{LPP} the authors proved a similar inequality
relating the dimension of the Brill-Noether locus of an algebraic
curve with the Brill-Noether rank of a tropical curve. The proof used
Baker's specialization lemma together with the following fact:
\begin{proposition}
(\cite[Proposition 5.1]{LPP}) Let $X$ be a smooth projective curve.
Suppose $W_{d}^{r}(X)$ is not empty, and let $E$ be an effective
divisor of degree $r+\dim(W_{d}^{r}(X))$ on $X$. Then there is a
divisor $D$ whose class is in $W_{d}^{r}(X)$ such that $D-E$ is
effective.
\end{proposition}
In what follows we shall use the above proposition together with the
specialization lemma for weighted graphs to do the same for the dual
weighted graph of the special fiber of a curve:
\begin{theorem}
\label{Specialization BN}Let $X$ be a smooth projective curve over
a discretely valued field with a regular semistable model whose special
fiber has weighted dual graph $\Gamma$. Then for every $r,d\in\mathbb{N}$,
$\dim W_{d}^{r}(X)\leq w_{d}^{r}(\Gamma)$. 
\end{theorem}
\begin{proof}
Denote the dimension of $W_{d}^{r}(X)$ by $\rho$. To show that the
above inequality is true we must show that every effective divisor
$E$ of degree $r+\rho$ is contained in an effective divisor $D$
of degree $d$ and rank at least $r$. Note that by definition, both
$D$ and $E$ are divisors supported on $\Gamma^{0}$. 

We start by showing that the statement is true whenever $E$ is a
rational divisor, namely that its distance from any vertex is rational.
Let $E$ be such a divisor. Then $E$ is the specialization of some
effective divisor $E^{X}$ of degree $r+\rho$ on $X$. Since $w_{d}^{r}(X)=\rho$,
it follows by \cite[Proposition 5.1]{LPP} that $E^{X}$ is contained
in an effective divisor $D^{X}$ of degree $d$ and rank at least
$r$, which specializes to a divisor $D$ that contains $E$. By the
Specialization Lemma, the rank of $D$ on $\Gamma$ is at least $r$.

Now, let $E$ be any effective divisor on $\Gamma^{0}$ and let $(E_{i})$
be a sequence of rational effective divisors converging to $E$. For
each $i$, choose a divisor $D_{i}$ of rank $r$ that contains
$E_{i}$. By compactness of $\mbox{Pic}_{d}(\Gamma)$, there is a
subsequence of $\{D_{i}\}$ which converges to a divisor $D$ that
contains $E$. Since each $D_{i}$ is in $W_{d}^{r}(\Gamma)$,
and $W_{d}^{r}(\Gamma)$ is closed, $D$ is of rank $r$
as well.
\end{proof}

\thanks{\textbf{Acknowldegments. }The author thanks S. Payne for suggesting the problem and for many helpful comments and discussions, and to the referee for his helpful remarks; the author gratefully acknowledges the hospitality provided by the Max Planck Institute for Mathematics (MPIM) during the spring of 2012 where much of this work was carried out. 
}

\

\noindent
 Yoav Len - yoav.len@yale.edu 
 
 \noindent
 Mathematics Department, Yale University
 
  \noindent
10 Hillhouse Ave, New Haven, CT, 06511 (USA)

\end{document}